%% file: main.tex
\documentclass[12pt,a4paper,leqno,noamsfonts]{amsart}
\usepackage[english]{babel}
\usepackage[dvipsnames]{xcolor}
\usepackage{graphicx,pifont}
\usepackage[utopia]{mathdesign}
\usepackage[a4paper,top=3cm,bottom=3cm,
           left=3.4cm,right=3.4cm,marginparwidth=60pt]{geometry}
\usepackage[utf8]{inputenc}
\usepackage{braket,caption,comment,mathtools,stmaryrd}
\usepackage[usestackEOL]{stackengine}
\usepackage{multirow,booktabs,microtype,relsize}
\usepackage[colorlinks,bookmarks]{hyperref} %
      \hypersetup{colorlinks,%
            citecolor=britishracinggreen,%
            filecolor=black,%
            linkcolor=cobalt,%
            urlcolor=cornellred}
      \numberwithin{equation}{section}
\usepackage[capitalise]{cleveref}

\input{macros.tex}

\title[A sign that used to annoy me, and still does]{A sign that used to annoy me, and still does}
\author{Andrea T. Ricolfi}

\begin{document}
\maketitle
\thispagestyle{empty}
\begin{abstract}
We provide a proof of the following fact: if a complex scheme $Y$ has  Behrend function constantly equal to a sign $\sigma \in \set{\pm 1}$, then all of its components $Z \subset Y$ are generically reduced and satisfy $(-1)^{\dim_{\BC} T_pY} = \sigma = (-1)^{\dim Z}$ for $p \in Z$ a general point. Given the recent counterexamples to the parity conjecture for the Hilbert scheme of points $\Hilb^n(\BA^3)$, our argument suggests a possible path to disprove the constancy of the Behrend function of $\Hilb^n(\BA^3)$. 
\end{abstract}

\section{Introduction}
By work of Behrend \cite{Beh}, every scheme $Y$ of finite type over $\BC$ carries a canonical constructible function $\nu_Y \colon Y(\BC) \to \BZ$, known as the \emph{Behrend function} of $Y$. It is a subtle invariant of singularities, with a key role in enumerative geometry. Already in the case of schemes with just one point, its computation is a nontrivial task \cite{Graffeo-Ricolfi}. The $\nu_Y$-weighted Euler characteristic of $Y$ is the global invariant
\[
\chi(Y,\nu_Y) = \sum_{m \in \BZ}m \chi(\nu_Y^{-1}(m))\,\in\,\BZ,
\]
where $\chi$ is the topological Euler characteristic. 

Fix $n \in \BZ_{\geq 0}$. Let $\HH_n = \Hilb^n(\BA^3)$ be the Hilbert scheme of $n$ points on affine 3-space, namely the moduli space of ideals $I \subset \BC[x,y,z]$ of colength $n$.
Behrend--Fantechi proved that 
\begin{equation}\label{eqn:signs}
(-1)^n = (-1)^{\dim_{\BC} T_I \HH_n} = \nu_{\HH_n}(I),
\end{equation}
as soon as $I$ is \emph{monomial} \cite{BFHilb}. Moreover, the main result of loc.~cit.~uses the above identities to compute
\begin{equation}\label{eqn:weighted-chi-Hilb}
\chi(\HH_n,\nu_{\HH_n}) = (-1)^n \chi(\HH_n).
\end{equation}
In other words, the $\nu_{\HH_n}$-weighted Euler characteristic of $\HH_n$, also known as the $n$-th degree $0$ \emph{Donaldson--Thomas invariant} of $\BA^3$, is the same that one would have if $\nu_{\HH_n}$ were constant.
One is then led to make the following prediction.

\begin{conj}\label{conj:behrend}
    The Behrend function of $\Hilb^n(\BA^3)$ is constantly equal to $(-1)^n$.
\end{conj}

A proof of this conjecture was proposed by A.~Morrison in \cite{Morrison-Behrend}, but a gap in the proof, to the best of our knowledge, has not been fixed since.

On the other hand, the first identity in Formula \eqref{eqn:signs} encourages the following conjecture, due to Okounkov--Pandharipande \cite{parity-conj-OP}.

\begin{conj}[Parity Conjecture]\label{conj:parity}
    One has $(-1)^n = (-1)^{\dim_{\BC} T_I \HH_n}$ for all $I \in \HH_n$.
\end{conj}

Maulik--Nekrasov--Okounkov--Pandharipande proved the parity conjecture for monomial ideals \cite[Thm.~2]{MNOP1} (see also \cite[Lemma 4.1\,(c)]{BFHilb}) and Ramkumar--Sammartano proved it for homogeneous ideals \cite[Thm.~1]{Ritvik-Sammartano}. However, we have the following recent result.

\begin{theorem}[Giovenzana--Giovenzana--Graffeo--Lella \cite{GGGL}]
\Cref{conj:parity} is false for $n \geq 12$.
\end{theorem}

\smallbreak
We shall prove the following.

\begin{theorem}[\Cref{thm:main-theorem--body}]\label{thm:main-theorem}
Let $Y$ be a scheme of finite type over $\BC$. Fix a sign $\sigma \in \set{\pm 1}$. If $\nu_Y \equiv \sigma$, then every irreducible component $Z \subset Y$ is generically reduced and a general point $p \in Z$ satisfies $\dim_{\BC} T_pY = \dim Z$ and
\[
(-1)^{\dim_{\BC} T_pY} = \sigma = (-1)^{\dim Z}.
\]
\end{theorem}

In particular, when $Y=\Hilb^n(\BA^3)$, \Cref{conj:behrend} implies a `generic' version of \Cref{conj:parity}. Such link between the two above conjectures was already stated in \cite{Morrison-Behrend}, but the proof of \Cref{thm:main-theorem} gives full details. Such details are needed to support the point of this paper, which is to provide a theoretical path (starting from the failure of \Cref{conj:parity}), to \emph{disprove} \Cref{conj:behrend}.

\smallbreak
Consider the Quot scheme of points $\mathrm{Q}_{r,n} = \Quot_{\BA^3}(\OO_{\BA^3}^{\oplus r},n)$, parametrising isomorphism classes of $\BC[x,y,z]$-linear quotients $\OO_{\BA^3}^{\oplus r} \onto T$, where $\dim_\BC T = n$. 
\Cref{eqn:signs} generalises to 
\begin{equation}\label{eqn:sign_higher-rank}
(-1)^{rn} = (-1)^{\dim_{\BC} T_p \mathrm{Q}_{r,n}} = \nu_{\mathrm{Q}_{r,n}}(p),
\end{equation}
for every $\mathbb G_m^{3+r}$-fixed point $p \in \mathrm{Q}_{r,n}$, i.e.~for $p$ corresponding to a direct sum of monomial ideals \cite{BR18,FMR_K-DT}.

If $\mathrm{Q}_{r,n}$ has constant Behrend function --- which cannot be excluded just yet --- then it cannot contain any generically nonreduced component. Finding such components, or even just nonreduced points, is an active research direction, for which we refer the reader to \cite{Jelisiejew-Pathologies,Jelisiejew-sivic,Szachniewicz}. The case of $\BA^3$ is in some sense the last mystery in the land of pathologies on Quot schemes of smooth varieties, for it finds itself sandwiched between smooth Hilbert schemes or mildly singular Quot schemes ($\BA^d$ case, for $d \leq 2$) and Quot schemes that happen to admit generically nonreduced components ($\BA^d$ case, for $d > 3$). 

Finally, the author wants to point out that the title of this work is in homage to P.~Tingley's paper \emph{A minus sign that used to annoy me but now I know why it is there} \cite{zbMATH06990163}. The ubiquity of signs in Donaldson--Thomas theory, and in particular the one in \Cref{eqn:sign_higher-rank}, has annoyed (read: fascinated) the author for some time. Even though we know the theoretical reasons \emph{why it is there}, we felt like writing this note to highlight the structural aspects of the theory, still to be understood, surrounding this sign.

\section{Local Euler obstruction and the Behrend function}
In this section we fix our notation and recall the definition of the Behrend function from \cite{Beh}.

All schemes are of finite type over $\BC$. The group of cycles on a scheme $Y$ is denoted $\mathrm{Z}_\ast Y$. An irreducible scheme $Y$, with generic point $\eta \in Y$, is \emph{generically reduced} if the local ring $\OO_{Y,\eta}$ is reduced. The \emph{multiplicity} $\mult_Z(Y)$ of an irreducible component $Z \subset Y$, with generic point $\eta$, is defined as the length of the local artinian ring $\OO_{Y,\eta}$. We assume all schemes admit a closed embedding in a smooth $\BC$-scheme; this assumption is never necessary, but is satisfied in all applications we have in mind.

\subsection{MacPherson's local Euler obstruction}
Let $Y$ be a scheme. The abelian group of constructible functions $Y(\BC) \to \BZ$ on a scheme $Y$ is denoted $\Con(Y)$. The \emph{local Euler obstruction} of $Y$ is an isomorphism of abelian groups
\[
\begin{tikzcd}
\Eu \colon \mathrm{Z}_\ast Y \arrow{r}{\sim} & \Con(Y),
\end{tikzcd}
\]
discovered by MacPherson \cite[Lemma 2]{MacPherson-Eu}. We now recall its purely algebraic definition, different from the original trascendental definition by MacPherson; see \cite{Jiang-Eu} for a modern survey on this topic. Let $V \into M$ be a closed immersion of a $d$-dimensional integral scheme $V$ inside a smooth scheme $M$. Let $\Gr_d(\CT_M) \to M$ be the Grassmann bundle of $d$-planes in the fibres of the tangent bundle $\CT_M$, and denote by $V_{\sm}\subset V$ the smooth locus of $V$, which is open and nonempty \cite[\href{https://stacks.math.columbia.edu/tag/056V}{Tag 056V}]{stacks-project}. We have a canonical section
\[
\mathsf s\colon V_{\sm} \to \Gr_d(\CT_M), \quad y \mapsto T_yV_{\sm}.
\]
Let $\widehat V \subset \Gr_d(\CT_M)$ be the closure of the image of $\mathsf s$. The map
\[
\mathsf n \colon \widehat V \to V
\]
restricting the projection $\Gr_d(\CT_M) \to M$ is called the \emph{Nash blowup} of $V\into M$. Its base change along $V_{\sm} \into V$ is an isomorphism. Let $\CU$ denote the universal rank $d$ bundle on $\Gr_d(\CT_M)$. The bundle $\widehat{\CT}_V = \CU|_{\widehat V}$ is called the \emph{Nash tangent bundle}.

Now back to defining $\Eu \colon \mathrm{Z}_\ast Y \to \Con(Y)$.
Let $V\subset Y$ be a prime cycle on $Y$, i.e.~a generator of $\mathrm{Z}_\ast Y$. 
Let $y \in Y$ be a closed point. The integer
\[
\Eu(V)(y) = \int_{\mathsf n^{-1}(y)} c(\widehat{\CT}_V) \cap s(\mathsf n^{-1}(y),\widehat V)
\]
is called the local Euler obstruction of $V$ at the point $y \in Y$. Here $s(\mathsf n^{-1}(y),\widehat V)$ denotes the Segre class of the normal cone to the closed immersion $\mathsf n^{-1}(y)\into \widehat V$. The map $\Eu$ is defined by $\BZ$-linear extension. By \cite[Section 3]{MacPherson-Eu}, $\Eu(V)(y)$ is equal to:
\begin{itemize}
    \item [$\circ$] $0$, if $y \in Y \setminus V$,
    \item [$\circ$] $1$, if $y$ is a nonsingular point on $V$,
    \item [$\circ$] $\mult_yY$, if $y$ is a closed point on an integral curve $Y=V$,
    \item [$\circ$]
    $d(2-d)$ is $V\subset \BA^3$ is the cone over a smooth degree $d$ plane curve $X \into \BP^2$ and $y \in V$ is the vertex of the cone.
\end{itemize}
The last bullet was generalised by Aluffi, who proved that if $X \subset \BP^{n-1}$ is a smooth curve of degree $d$ and genus $g$, and if $y$ is the vertex of the cone $V \subset \BP^n$ over $X$, then $\Eu(V)(y) = 2-2g-d$ \cite[Ex.~3.19]{zbMATH06823248}. 

\subsection{The Behrend function}
Behrend proved that any finite type $\BC$-scheme $Y$ carries a canonical cycle $\mathfrak c_Y \in \mathrm{Z}_\ast Y$, whose definition we now briefly recall.

First of all, suppose given a scheme $U$ and a closed immersion $U \into M$ inside a smooth scheme $M$, cut out by the ideal sheaf $\mathscr I \subset \OO_M$. Consider the normal cone
\[
\begin{tikzcd}
    C_{U/M} = \mathbf{Spec}_{\OO_U} \left(\displaystyle\bigoplus_{e \geq 0}\mathscr I^e/\mathscr I^{e+1}\right) \arrow{r}{\pi} & U.
\end{tikzcd}
\]
Note that if $D$ is an irreducible component of $C_{U/M}$, then $\pi(D)$ is an irreducible closed subset of $X$, and as such it defines a cycle $\pi(D) \in \mathrm{Z}_\ast U$.
The \emph{signed support of the intrinsic normal cone}, introduced by Behrend in \cite[Sec.~1.1]{Beh}, is the cycle
\[
\mathfrak c_{U/M} = \sum_{D \subset C_{U/M}} (-1)^{\dim \pi(D)} \mult_D(C_{U/M})\cdot\pi(D) \in \mathrm{Z}_\ast U,
\]
the sum being over the irreducible components of the normal cone $C_{U/M}$.

Now back to our scheme $Y$. The canonical cycle $\mathfrak c_Y \in \mathrm{Z}_\ast Y$ is defined as follows: it is the unique cycle with the property that for any \'etale map $U \to Y$ and for any closed immersion $U \into M$ inside a smooth scheme, one has $\mathfrak c_{U/M} = \mathfrak c_Y|_U$. See \cite[Prop.~1.1]{Beh} for the proof that this is a good definition. In particular, if $Y$ itself admits a closed immersion inside a nonsingular scheme $M$, then $\mathfrak c_Y = \mathfrak c_{Y/M}$. 

\begin{definition}[{\cite[Def.~1.4]{Beh}}]
Let $Y$ be a scheme of finite type over $\BC$. The \emph{Behrend function} of $Y$ is the constructible function $\nu_Y = \Eu(\mathfrak c_Y)$.
\end{definition}

\begin{example}
If $Y$ is a smooth connected scheme of dimension $d$, then $\mathfrak c_Y = (-1)^{d}[Y]$, so that $\nu_Y \equiv (-1)^{d}$. The Behrend function of a fat point --- a scheme with only one point --- is also trivially constant, but hard to compute in general (e.g.~for embedding dimension higher than $3$), see \cite{Graffeo-Ricolfi}. 
\end{example}

\begin{remark}\label{rmk:nu-for-irr-cpt}
The Behrend function pulls back along \'etale maps, in particular $\nu_Y(p) = \nu_U(p)$ if $U$ is open in $Y$ and $p \in U$.
Let $Y$ be a scheme, $Z \subset Y$ an irreducible component, $p \in Z$ a closed point. Assume there is an open subset $U \subset Z$ containing $p$, and not intersecting any other irreducible component of $Y$. Then $U$ is also open in $Y$, so
\[
\nu_Z(p) = \nu_U(p) = \nu_Y(p).
\]
\end{remark}

Assume $Y$ is an irreducible scheme admitting a closed immersion into a smooth scheme $M$. Let $C=C_{Y/M}$ be the normal cone, with projection $\pi\colon C \to Y$, and write $\overline D = \pi(D)$ for an irreducible component $D$ of $C$. 

Let $p \in Y$ be a point. Then, since $\nu_Y = \Eu(\mathfrak c_{Y/M})$, we have
\begin{equation}\label{eqn:nu1}
\nu_Y(p) = 
\sum_{\substack{D\subset C \\ \overline D \ni p}}(-1)^{\dim \overline{D}}\mult_D(C)\cdot \Eu(\overline{D})(p).
\end{equation}
There are two possibilities for each $D$ in the sum: either $\overline{D}=Y$, or $\overline{D}\neq Y$. Therefore Formula \eqref{eqn:nu1} becomes
\begin{multline}\label{eqn:nu3}
\nu_Y(p) = \sum_{\substack{D\subset C \\ Y \neq \overline D \ni p}}(-1)^{\dim \overline{D}}\mult_D(C)\cdot \Eu(\overline{D})(p) \\
+ (-1)^{\dim Y} \Eu(Y)(p) \sum_{\substack{D \subset C \\ \overline D = Y}}\mult_D(C).
\end{multline}

We conclude this section with a few examples.

\begin{example}
Here we give an example of an irreducible scheme $Y$ which is generically reduced, but whose Behrend function is not constantly equal to the same sign. Take $Y = \Spec \BC[x,y]/(y^2,xy)\subset \BA^2$. It is smooth everywhere except at the origin $0 \in Y \subset \BA^2$, where an embedded point is located. We have $\nu_Y|_{Y \setminus 0} \equiv -1$, but $\nu_Y(0) = 1$. Indeed, as computed in Examples 3.5 and 4.8 of \cite{VFC_working_math}, the normal cone $C=C_{Y/\BA^2}$ has two irreducible components $D_1$ and $D_2$, where $D_1$ (called $C_0$ in loc.~cit.) has multiplicity $2$ and is contracted by $\pi \colon C \to Y$ onto $0 \in Y$, and $D_2$ (called $L$ in loc.~cit.) has multiplicity $1$ and dominates $Y$. We therefore have
\begin{multline*}
\nu_Y(0) 
= (-1)^0 \mult_{D_1}(C)\Eu(\overline D_1)(0) + (-1)^{\dim Y} \mult_{D_2}(C)\Eu(\overline D_2)(0) \\ = 1\cdot 2\cdot 1 + (-1)\cdot 1 \cdot 1 = 1.
\end{multline*}
\end{example}

\begin{example}\label{ex:3lines}
Here we give an example of a singular, reducible and reduced scheme $Y$ with constant Behrend function. Let $Y \subset \BA^3$ be the union of the three coordinate axes in $\BA^3$, namely $Y = \Spec \BC[x,y,z]/J$ where $J = (xy,xz,yz)$. We claim that $\nu_Y \equiv -1$. The only nontrivial value to compute is $\nu_Y(0)$, where $0$ denotes the origin of $\BA^3$. We call $L_1,L_2,L_3$ the coordinate axes, and we set $f_0=xy$, $f_1=xz$ and $f_2 = yz$. Fix homogeneous coordinates $w_0,w_1,w_2$ on $\BP^2$, and form the blowup
\[
\varepsilon \colon \Bl_{Y}\BA^3 \into \BA^3 \times \BP^2 \to \BA^3,
\]
which is cut out by the ideal
\[
\mathscr I = (w_2f_0-w_1f_1,w_2f_0-w_0f_2) \subset \OO_{\BA^3 \times \BP^2}.
\]
Therefore the exceptional divisor $E_{Y}\BA^3 = V(\varepsilon^{-1}(J)\cdot \OO_{\Bl_{Y}\BA^3})$ is cut out by the relations
\[
w_2f_0=w_1f_1,\quad w_2f_0=w_0f_2,\quad xy=xz=yz=0.
\]
From these, one can check that $E_{Y}\BA^3$ consists of irreducible components $E_1,\ldots,E_4$ with $E_i \cong \BA^1 \times \BP^1$ for $i=1,2,3$, and $E_4 \cong \BP^2$. The multiplicities are $\mult_{E_i}E_{Y}\BA^3 = 1$ for $i=1,2,3$ and $\mult_{E_4}E_{Y}\BA^3 = 2$. The component $E_i$ dominates $L_i$ for $i=1,2,3$, whereas $E_4$ is collapsed onto $\set{0}$. Therefore
\[
\nu_Y(0) = 3 \cdot (-1)^{\dim L_1}\cdot 1\cdot \Eu(L_1)(0) + (-1)^0\cdot 2 \cdot \Eu(\set{0})(0) = -3+2 = -1.
\]
See also \Cref{conclusions:crit-xyz} for one more comment on this example.
\end{example}

\begin{example}
Let $Y = \Hilb^4(\BA^3)$ be the Hilbert scheme of $4$ points on $\BA^3$. It is an irreducible scheme of dimension $12$, with singular locus equal to the closed subscheme $\BA^3 \subset Y$ parametrising the squares $\mathfrak m_x^2$ of the maximal ideals of closed points $x \in \BA^3$. This is the `first' singular Hilbert scheme. Its Behrend function \emph{is} in fact constant, 
\[
\nu_Y \equiv 1.
\]
Since the Behrend function of a smooth irreducible scheme $R$ is constantly equal to $(-1)^{\dim R}$, we only need to check that $\nu_Y(\mathfrak m_x^2) = 1$ for all $x \in \BA^3$. In fact, each point $\mathfrak m_x^2$ is a translation of the monomial point $I=\mathfrak m_0^2$ (the only singular monomial ideal), so it is enough to confirm that $\nu_Y(I) = 1$. But this follows from \cite[Thm.~3.4]{BFHilb}, which implies
\[
\nu_Y(I) = (-1)^{\dim_{\BC}T_{I}Y} = (-1)^{18} = 1.
\]
\end{example}

\section{Proof of the main theorem}

In this section we prove \Cref{thm:main-theorem}. 

\begin{theorem}\label{thm:main-theorem--body}
Let $Y$ be a scheme of finite type over $\BC$. Fix a sign $\sigma \in \set{\pm 1}$. If $\nu_Y \equiv \sigma$, then every irreducible component $Z \subset Y$ is generically reduced and a general point $p \in Z$ satisfies $\dim_{\BC} T_pY = \dim Z$ and
\begin{equation}\label{eqn:two-formulas}
    (-1)^{\dim_{\BC} T_pY} = \sigma = (-1)^{\dim Z}.
\end{equation}
\end{theorem}

\begin{proof}
Let $Z \subset Y$ be an irreducible component. Consider the open subset 
\begin{equation}\label{eqn:open-away-from-components}
W = Y \setminus \bigcup_{Z'\neq Z}Z' \subset Y,
\end{equation}
the union being over the irreducible components $Z' \subset Y$ different from $Z$. Then for $p\in W$, one has
\begin{equation}\label{eqn:nu7}
\sigma = \nu_Y(p) = \nu_Z(p), 
\end{equation}
the second identity being implied by \Cref{rmk:nu-for-irr-cpt}.

Let $Y \into M$ be a closed immersion inside a smooth scheme $M$. Let $C = C_{Z/M}$ be the normal cone to $Z$ in $M$, with projection $\pi \colon C \to Z$. Let $W' \subset Z$ be the (open) complement of the closed subset
\[
\bigcup_{\substack{D\subset C \\ \overline{D} \neq Z}} \overline{D} \subset Z,
\]
where $\overline{D}=\pi(D)$ as before. Note that $W'$ might equal $Z$, but cannot be empty. Then, for $p \in W \cap W'$, one has
\begin{equation} \label{eqn:nu6}
\nu_Z(p) = (-1)^{\dim Z} \Eu(Z)(p) \sum_{\substack{D\subset C \\ \overline{D}=Z}}\mult_D(C)
\end{equation}
by Formula \eqref{eqn:nu3}. However, we have
\[
\Eu(Z)(p) = \Eu(Z_{\red})(p),
\]
and $Z_{\red}$ is a reduced scheme of finite type over $\BC$, therefore its smooth locus $W''$ is a dense open subset. Thus $\Eu(Z)(p) = 1$ for $p \in W''$, and Formula \eqref{eqn:nu6} becomes
\begin{equation}\label{eqn:nu5}
\nu_Z(p) = (-1)^{\dim Z}\sum_{\substack{D \subset C \\ \overline D = Z}}\mult_D(C),\quad p \in W \cap W' \cap W''.
\end{equation}
Thus, combining Formula \eqref{eqn:nu5} and Formula \eqref{eqn:nu7} with one another, we find
\begin{equation}\label{eqn:nu4}
\sigma = (-1)^{\dim Z}\cdot \mathsf m,
\end{equation}
where $\mathsf m \in \BZ_{>0}$ is the sum appearing in Formula \eqref{eqn:nu5}. But this forces $\mathsf m = 1$ and thus $(-1)^{\dim Z} = \sigma$, too. This proves the second identity in Formula \eqref{eqn:two-formulas}.

Next, we prove that $Z$ is generically reduced. Since $\mathsf m = 1$, there is a unique irreducible component $D \subset C$ such that $\overline{D} = Z$, and moreover $C$ is reduced at the generic point $\xi_D \in C$ corresponding to $D\subset C$, since necessarily $\mult_D(C)=1$. It follows that $Z$ is generically reduced. In a little more detail, consider the composition $\id_Z = \pi\circ\tau \colon Z \to C \to Z$ where $\tau\colon Z \to C$ is the zero section of the cone. Now, $\pi$ maps $\xi_D \in C$ to the generic point $\eta$ of $Z$. Since the composition 
\[
\OO_{Z,\eta} \to \OO_{C,\xi_D} \to \OO_{Z,\eta}
\]
is the identity, the map $\OO_{Z,\eta} \to \OO_{C,\xi_D}$ is injective. But $\OO_{C,\xi_D}$ is reduced, thus $\OO_{Z,\eta}$ is reduced, too.

Finally, we prove the first identity in Formula \eqref{eqn:two-formulas}. By the previous paragraph, there is a nonempty open reduced subscheme $U \subset Z$. In particular, $U$ contains a nonempty smooth open subset $U'\subset Z$. Consider also the open subset $W$ (open in both $Z$ and $Y$) from \Cref{eqn:open-away-from-components}. Then, for any point $p \in U' \cap W$, we have
\begin{align*}
    \dim Z
    &= \dim U' & Z\mbox{ is irreducible} \\
    &= \dim_{\BC} T_pU' & U'\mbox{ is smooth} \\
    &= \dim_{\BC} T_pZ & U' \mbox{ is open in }Z\\
    &= \dim_{\BC} T_pW & W \mbox{ is open in }Z \\
    &= \dim_{\BC} T_pY & W \mbox{ is open in }Y
\end{align*}
which finishes the proof.
\end{proof}

Fix integers $r,n \in \BZ_{>0}$. Consider the Quot scheme of points $\Quot_{\BA^3}(\OO_{\BA^3}^{\oplus r},n)$, parametrising (isomorphism classes of) length $n$ quotients $\OO_{\BA^3}^{\oplus r} \onto T$. This is the general situation we have in mind. Note that this Quot scheme has a natural embedding in a smooth quasiprojective variety $\ncQuot^n_{r}$ of dimension $2n^2+rn$, the so-called noncommutative Quot scheme \cite[Thm.~2.6]{BR18}. We obtain the following consequence of \Cref{thm:main-theorem}.

\begin{corollary}\label{cor:disprove-nu}
Let $Z \subset \Quot_{\BA^3}(\OO_{\BA^3}^{\oplus r},n)$ be an irreducible component. If either $Z$ is not generically reduced, or 
    \[
    (-1)^{\dim Z} \neq (-1)^{rn},
    \]
then the Behrend function of $\Quot_{\BA^3}(\OO_{\BA^3}^{\oplus r},n)$ is not constant.
\end{corollary}

\section{Conclusions}\label{conclusions:crit-xyz}
The falsity of the parity conjecture is unfortunately not enough to disprove the constancy of the Behrend function of $\Hilb^n(\BA^3)$. However, general enough counterexamples to the parity conjecture would constitute a good testing ground for disproving \Cref{conj:behrend}. Our main argument says precisely that points as general as in \Cref{eqn:nu5} would be enough to disprove it. 

\smallbreak
We believe, nevertheless, that having constant Behrend function equal to a sign should be thought of as a sort of `regularity property'. In fact, \Cref{conj:behrend} is a pretty subtle statement, precisely because the Hilbert scheme of points $\Hilb^n(\BA^3)$ does exhibit such regularity (or symmetry): it is a global \emph{critical locus}, i.e.~the zero scheme of an exact $1$-form $\dd f_n$ for $f_n$ a regular function on a smooth $\BC$-scheme $U_n$, see e.g.~\cite[Prop.~1.3.1]{MR2403807} or \cite[Thm.~2.6]{BR18}. We take \Cref{ex:3lines} as a toy situation to further explain this point. 
In that example, $Y = \Spec \BC[x,y,z] / (xy,xz,yz) \subset \BA^3$ is a critical locus, 
\[
Y = \crit(xyz),
\]
and as such it has a symmetric perfect obstruction theory in the sense of \cite{BFHilb}. Via the perfect obstruction theory machinery, one can compute directly
\[
\nu_Y(0) = (-1)^{\dim_{\BC}T_0Y-\dim_{\BC}T_0Y^{\BC^\times}} \cdot \nu_{Y^{\BC^\times}}(0) = (-1)^{3-0} \cdot 1 = -1,
\]
thanks to \cite[Thm.~C]{zbMATH06221964}, using that $Y$ has a $\BC^\times$-action and the symmetric perfect obstruction theory is $\BC^\times$-equivariant. This confirms once more that $\nu_Y \equiv -1$. Now, \Cref{ex:3lines} shows \emph{directly} how signed multiplicities in the normal cone for this critical locus balance each other and produce a constant Behrend function. It is possible that this behaviour does in fact occur for $\Hilb^n(\BA^3)$ as well, but checking this directly seems unfortunately out of reach at the moment.

\subsection*{Acknowledgements}
We thank Michele Graffeo for the numerous conversations around the parity conjecture, and for pointing out \Cref{ex:3lines}. We thank Martijn Kool and Sergej Monavari for all the instructive discussions around the Behrend function throughout the years.

\bibliographystyle{amsplain-nodash}
\bibliography{bib}

\bigskip
\bigskip
\bigskip
\noindent
Andrea T. Ricolfi \\
\address{SISSA, Via Bonomea 265, 34136, Trieste (Italy)} \\
\texttt{aricolfi@sissa.it}

\end{document}

%% file: macros.tex
\DeclareSymbolFont{usualmathcal}{OMS}{cmsy}{m}{n}
\DeclareSymbolFontAlphabet{\mathcal}{usualmathcal}
\DeclareMathAlphabet\BCal{OMS}{cmsy}{b}{n}

\definecolor{cornellred}{rgb}{0.7, 0.11, 0.11}
\definecolor{britishracinggreen}{rgb}{0.0, 0.26, 0.15}
\definecolor{cobalt}{rgb}{0.0, 0.28, 0.67}

\usepackage[colorinlistoftodos]{todonotes} 
\usepackage{amsmath,float}
\usetikzlibrary{decorations.markings}

\newcommand{\Eu}{\mathsf{Eu}}
\newcommand{\Bl}{\mathsf{Bl}}
\newcommand{\Con}{\mathbf{F}}
\newcommand{\Gr}{\mathrm{Gr}}
\newcommand{\sm}{\mathrm{sm}}
\newcommand{\crit}{\mathrm{crit}}
\newcommand{\dd}{\mathrm{d}}
\newcommand{\mult}{\mathrm{mult}}

\newcommand{\simto}{\,\widetilde{\to}\,}
\newcommand{\into}{\hookrightarrow}
\newcommand{\onto}{\twoheadrightarrow}
\newcommand{\HH}{\mathrm{H}}
\newcommand{\OO}{\mathscr O}

\DeclareMathOperator{\Hilb}{Hilb}
\DeclareMathOperator{\ncQuot}{ncQuot}
\DeclareMathOperator{\Quot}{Quot}
\DeclareMathOperator{\Spec}{Spec}
\DeclareMathOperator{\id}{id}
\DeclareMathOperator{\red}{red}

\newcommand{\BA}{{\mathbb{A}}}
\newcommand{\BC}{{\mathbb{C}}}
\newcommand{\BP}{{\mathbb{P}}}
\newcommand{\BZ}{{\mathbb{Z}}}
\newcommand{\CT}{{\mathcal{T}}}

\newcommand{\CU}{{\mathcal{U}}}

\usepackage[all]{xy}
\usepackage{tikz}
\usepackage{tikz-cd}
\usepackage{adjustbox}
\usepackage{rotating}
\usepackage{comment}

\usetikzlibrary{matrix,shapes,intersections,arrows,decorations.pathmorphing}
\tikzset{commutative diagrams/arrow style=math font}
\tikzset{commutative diagrams/.cd,
mysymbol/.style={start anchor=center,end anchor=center,draw=none}}

\tikzset{
shift up/.style={
to path={([yshift=#1]\tikztostart.east) -- ([yshift=#1]\tikztotarget.west) \tikztonodes}
}
}

\theoremstyle{definition}

\newtheorem*{lemma*}{Lemma}
\newtheorem*{theorem*}{Theorem}
\newtheorem*{example*}{Example}
\newtheorem*{fact*}{Fact}
\newtheorem*{notation*}{Notation}
\newtheorem*{definition*}{Definition}
\newtheorem*{prop*}{Proposition}
\newtheorem*{remark*}{Remark}
\newtheorem*{corollary*}{Corollary}

\newtheorem*{conventions*}{Conventions}

\newtheorem{definition}{Definition}[section]

\newtheorem{example}[definition]{Example}

\newtheorem{remark}[definition]{Remark}

\newtheoremstyle{thm} 
        {3mm}
        {3mm}
        {\slshape}
        {0mm}
        {\bfseries}
        {.}
        {1mm}
        {}
\theoremstyle{thm}
\newtheorem{theorem}[definition]{Theorem}
\newtheorem{corollary}[definition]{Corollary}

\newtheorem{conj}{Conjecture}

\newtheoremstyle{ex} 
        {3mm}
        {3mm}
        {}
        {0mm}
        {\scshape}
        {.}
        {1mm}
        {}
\theoremstyle{ex}

\newtheoremstyle{sol} 
        {3mm}
        {3mm}
        {}
        {0mm}
        {\scshape}
        {.}
        {1mm}
        {}
\theoremstyle{sol}

\newtheorem*{Acknowledgments*}{Acknowledgments}